\documentclass[11pt]{article}

\usepackage[left=2.cm, right=2.cm, top=2.5cm, bottom=2.5cm]{geometry}
\usepackage{amsmath,amssymb,amsfonts,amsthm,mathtools}
\usepackage[T1]{fontenc}
\usepackage{lmodern}
\usepackage{microtype}
\usepackage{enumitem}
\usepackage{aliascnt}
\usepackage{xcolor}
\usepackage[colorlinks=true,linkcolor=blue!55!black,citecolor=blue!55!black,urlcolor=blue!55!black]{hyperref}
\usepackage[nameinlink,capitalize]{cleveref}
\numberwithin{equation}{section}
\usepackage{authblk}

\newtheorem{theorem}{Theorem}[section]
\newaliascnt{proposition}{theorem}
\newtheorem{proposition}[proposition]{Proposition}
\aliascntresetthe{proposition}
\newaliascnt{lemma}{theorem}
\newtheorem{lemma}[lemma]{Lemma}
\aliascntresetthe{lemma}
\theoremstyle{definition}
\newaliascnt{definition}{theorem}
\newtheorem{definition}[definition]{Definition}
\aliascntresetthe{definition}
\newaliascnt{example}{theorem}
\newtheorem{example}[example]{Example}
\aliascntresetthe{example}
\theoremstyle{remark}
\newaliascnt{remark}{theorem}
\newtheorem{remark}[remark]{Remark}
\aliascntresetthe{remark}
\crefname{theorem}{Theorem}{Theorems}
\Crefname{theorem}{Theorem}{Theorems}
\crefname{proposition}{Proposition}{Propositions}
\Crefname{proposition}{Proposition}{Propositions}
\crefname{lemma}{Lemma}{Lemmas}
\Crefname{lemma}{Lemma}{Lemmas}
\crefname{definition}{Definition}{Definitions}
\Crefname{definition}{Definition}{Definitions}
\crefname{example}{Example}{Examples}
\Crefname{example}{Example}{Examples}
\crefname{remark}{Remark}{Remarks}
\Crefname{remark}{Remark}{Remarks}

\newcommand{\E}{\mathbb E}
\newcommand{\Pp}{\mathbb P}
\newcommand{\C}{\mathbb C}
\newcommand{\R}{\mathbb R}
\newcommand{\Tr}{\operatorname{Tr}}
\newcommand{\op}{\mathrm{op}}
\newcommand{\msc}{m_{\mathrm{sc}}}
\newcommand{\musc}{\mu_{\mathrm{sc}}}
\newcommand{\sigstar}{\sigma_N^*}
\newcommand{\1}{\mathbf 1}
\newcommand{\dd}{\,\mathrm d}
\newcommand{\norm}[1]{\left\lVert #1\right\rVert}
\newcommand{\ip}[2]{\left\langle #1,#2\right\rangle}
\newcommand{\bigp}[2]{\big\langle #1,#2\big\rangle}

\newcommand{\xtoP}{\xrightarrow{\Pp}}

\title{BBP Transition for Deformed Inhomogeneous Random Matrices}

\author[1]{Ruohan Geng}
\author[2]{Zhantao He}
\author[3]{Dang-Zheng Liu}
\author[4]{Guangyi Zou}

\affil[1]{Institut de Math\'ematiques de Marseille, Aix-Marseille Universit\'e, \texttt{ruohan.GENG@univ-amu.fr}}
\affil[2]{School of Mathematical Sciences, Peking University, \texttt{zthe25@stu.pku.edu.cn}}
\affil[3]{School of Mathematical Sciences, University of Science and Technology of China, \texttt{dzliu@ustc.edu.cn}}
\affil[4]{Department of Mathematics, University of California, Irvine, \texttt{zouguangyi2001@gmail.com}}
\date{}

\begin{document}
\maketitle

\begin{abstract}

We study finite-rank deformations of inhomogeneous random symmetric/Hermitian matrices
 $H_N=(\sigma_{ij}W_{ij})_{i,j=1}^N,$
whose variance profile satisfies
\(
\sum_j\sigma_{ij}^2=1.
\)
Writing \(
\sigma_N^*=\max_{i,j}\sigma_{ij}
\) for the matrix sparsity parameter,
we establish the BBP transition of outlier eigenvalues at the level of the law of large numbers in  two regimes of moment conditions on $\{W_{ij}\}$:
(1)
under a uniform sub-Weibull bound with exponent \(0<\alpha\le 2\) and the sharp assumption
     $  \sigma_N^*(\log N)^{1/\alpha}\to 0,$
    the convergence holds almost surely;
   (2) under uniformly integrable $p$-th powers for some \(p\ge 4\) and  the sharp assumption
   $N(\sigma_N^*)^{p-2}=O(1),$
    the convergence holds in probability.
 Our approach combines sharp norm estimates, fixed-degree polynomial isotropic laws, and a finite-dimensional resolvent equation.
\end{abstract}

\section{Introduction}\label{sec:intro}
\subsection{Deformed Wigner matrices}

In his Selected Papers of Freeman Dyson with Commentary \cite{dyson1996selected}, Dyson remarked that the physical motivation for introducing matrix-valued Brownian motion in \cite{dyson1962brownian} is that it represents a system whose Hamiltonian is the sum of two parts, one known and one unknown. This observation serves as the starting point for deformed random matrix models, in which a Hamiltonian is described as the sum of a known deterministic part and a random part. A canonical example is obtained by adding independent Gaussian noise to a Wigner matrix, yielding a Gaussian-divisible ensemble that, conditional on the initial matrix, is a deterministic deformation of a Gaussian ensemble. Such models have played a key role in the study of universality in random matrix theory \cite{erdHos2010bulk,johansson2001universality,tao2011random}, and they also exhibit outlier and phase transition phenomena \cite{baik2005phase}. These features highlight the richness of eigenvalue statistics in random matrix theory and have attracted considerable interest owing to their broad applications in statistics, mathematical physics, random graphs, and signal processing. For a comprehensive overview, see the excellent survey \cite{peche2014deformed}.

To illustrate the outlier phenomenon in deformed random matrix models, we consider a rank-one perturbation of a GOE matrix 
\[
 X_N^{\mathrm{GOE}}=H_N^{\mathrm{GOE}}+\frac{a}{N}\1_N\1_N^*,
 \qquad a>0,
\]
where  $H_N^{\mathrm{GOE}}$ denotes  the $N\times N$ normalized Gaussian
orthogonal ensemble with off-diagonal variance $1/N$, and  $\1_N$ is the column vector of ones.  
Let \(\lambda_1(X_N)\ge\cdots\ge\lambda_N(X_N)\) be the ordered eigenvalues of \(X_N^{\mathrm{GOE}}\). 
  The combination of the  
 Wigner's semicircle law~\cite{Wigner} and the finite-rank inequality~\cite[Theorem~A.43]{BS} shows that its empirical spectral measure,  defined by 
\[
  \mu_{X_N}:=\frac1N\sum_{j=1}^N\delta_{\lambda_j(X_N)},
\]
still converges weakly, almost surely, to
\begin{equation}\label{eq:semicircle}
 \musc(\dd x)=\frac{1}{2\pi}\sqrt{4-x^2}\, \1_{\{|x|\leq2\}}\dd x.
\end{equation}
However, when 
$a
>1$, such a deformation may create a single outlier eigenvalue separating from the bulk spectrum, located at the value
$a+{1}/{a}$.
 This exact formula was first derived in  \cite{jones1978eigenvalue}, and was later independently rediscovered in  \cite{furedi1981eigenvalues} in a more general setting motivated by applications in random graphs.
   For exact results on rank-one perturbations, and for a detailed review of the integrable structures hidden in rank-one perturbed random matrices, we refer to \cite{forrester2023rank}.

The transition for outlier eigenvalues was first established for spiked complex Wishart ensembles with arbitrary finite-rank perturbations, and is now well known as the BBP phase transition after the seminal work of Baik, Ben Arous, and P\'ech\'e \cite{baik2005phase}, both at the level of the law of large numbers and of fluctuations. The additive counterpart and the critical transition for deformed Gaussian and spiked Wishart ensembles were studied in \cite{peche2006largest,bloemendal2013limits,bloemendal2016limits}. We refer to \cite{CDF,capitaine2012central,benaych2011fluctuations} and the references therein for more results on fluctuations. Beyond Gaussian ensembles, the study of finite-rank deformations has developed through high trace moments \cite{FeralPeche2007}, integral transforms of the limiting spectral measure \cite{BGN}, and isotropic resolvent estimates \cite{knowles2013isotropic,knowles2014outliers}. Fluctuation results cover random perturbations with delocalized eigenvectors \cite{benaych2011fluctuations}, deterministic finite-rank perturbations \cite{pizzo2013finite}, and delocalized spike directions \cite{renfrew2013finite}. For spiked covariance matrices, almost-sure eigenvalue limits and central limit theorems are available in \cite{baik2006eigenvalues,bai2008central}; information-plus-noise and rectangular models are treated in \cite{loubaton2011almost,capitaine2014exact,benaych2012singular,ding2022edge}.

\subsection{Inhomogeneous random matrices}
Inhomogeneous random matrices admit a general variance profile and encompass random band matrices as well as Wigner matrices restricted to regular graphs. Such models not only possess richer mathematical structure but are also indispensable for modeling real-world data and physical systems—spanning functional analysis \cite{davidson2001local} and mathematical physics \cite{spencer2011random}, through applied mathematics ranging from numerical linear algebra to computer and data science—where random matrices are often strongly inhomogeneous and devoid of inherent symmetries \cite{bandeira2025topics,tropp2015introduction,vershynin2018high}. Consequently, developing a general theory characterizing the spectral properties of arbitrarily structured inhomogeneous random matrices has become a central challenge.
Recently, inhomogeneous random matrices have attracted substantial interest, with major advances concerning spectral norms \cite{LVHY,van2017structured,BvH,adamczak2024norms}, matrix concentration and deviation inequalities \cite{MR4635836,brailovskaya2024extremal,MR4823211}, and spectral outliers \cite{bandeira2024matrixconcentrationinequalitiesfree,MR4234995}. For the study of local spectral properties, geometry and sparsity constitute two fundamental structural features of these models; we refer the reader to the surveys \cite{spencer2011random,van2017structured,bourgade2018random}. In particular, for random band matrices, the extreme eigenvalues depend on the bandwidth and spatial dimension \cite{sodin2010spectral,liu2023edge}, demonstrating that the variance profile can influence local spectral behavior.

Specifically, our object of study is the model \(H_N=(\sigma_{ij}W_{ij})\), where the variance profile satisfies \(\sum_j\sigma_{ij}^2=1\) for all \(i\). We write the sparsity measure \(\sigma_N^*:=\max_{i,j}\sigma_{ij}\). This parameter measures the largest entry scale relative to the total row variance. Independence, centering, asymptotically unit row variances, and a Lindeberg condition give the semicircle law when \(\sigma_N^*\to0\); see \cite{gotze2015limit,Chin}.
This global limit alone does not control isolated eigenvalues. For real symmetric matrices with sub-Gaussian entries, a sharp criterion for the absence of   outliers is established   in \cite[Theorem~1.2]{ASTY}:
\[
 \sigma_N^*\sqrt{\log N}\longrightarrow0
 \quad\Longrightarrow\quad
 \norm{H_N}_{\op}\longrightarrow2\qquad\text{almost surely}.
\]
For sub-Weibull entries with \(0<\alpha\le2\), \cite[Remark~2.2]{ASTY} gives the undeformed norm limit under \(\sigma_N^*(\log N)^{1/\alpha}\to0\). Earlier edge bounds for band and sparse matrices appear in \cite{khorunzhiy2008estimates,benaych2014largest,sodin2010spectral}.

For finite-rank deformations of periodic band matrices, Au \cite[Theorem~1.3]{Au} proved the almost-sure eigenvalue and eigenvector BBP transition for periodic band matrices with arbitrary deterministic spike directions. The assumptions are uniform bounds on all fixed-order moments of the entries and a bandwidth \(W\gg N^{\varepsilon}\) for some \(\varepsilon>0\). For Gaussian band matrices, convergence in probability also holds when the bandwidth grows faster than \(\log N\) \cite[Remark~1.4]{Au}. However, a natural question is whether, and how, the entry tails affect the BBP law of large numbers when the perturbation rank is held fixed.
For uniformly weighted \(k\)-regular supports, the direct application in \cite[Example~3.2]{bandeira2024matrixconcentrationinequalitiesfree} gives the rank-one eigenvalue and eigenvector BBP limits for uniformly bounded entries under \(k\gg(\log N)^4\).

\subsection{Model and main results}\label{sec:model}

In this paper, we study real symmetric or complex Hermitian inhomogeneous random matrices subject to a deterministic finite-rank deformation. 
\begin{definition}[Deformed inhomogeneous random matrices]\label{def:profile}
For each $N$, let 
\begin{align}
 H_N&=\Sigma_N\circ W_N,\label{eq:model-H}\\
 X_N&=H_N+A_N,\label{eq:deformation}
\end{align}
where $\circ$ denotes entrywise product, and the three   $N\times N$ matrices  $W_N$, $\Sigma_N$, and $A_N$  satisfy the following assumptions.
\begin{enumerate}[label=\textbf{(A\arabic*)},ref=A\arabic*,leftmargin=3.4em]
\item\label{ass:noise} \textbf{Noise.} The matrix $W_N=(W_{ij})$ is real symmetric or complex Hermitian. Its entries on and above the diagonal are independent and centered. In the real case,
\[
 \E\big[W_{ii}^2\big]=2,\qquad \E\big[W_{ij}^2\big]=1\quad(i<j).
\]
In the complex case, the diagonal entries are real and
\[
 \E\big[W_{ii}^2\big]=1,\qquad
 \E\big[|W_{ij}|^2\big]=1\quad(i<j).
\]
The entry laws may depend on $N,i,j$ and need not be identical or symmetric in distribution. 

\item\label{ass:profile} \textbf{Variance profile.} The matrix $\Sigma_N=(\sigma_{ij})$ is deterministic, symmetric and nonnegative. Its squared entries form a stochastic matrix:
\begin{equation}\label{eq:stochastic}
 \sum_{j=1}^N\sigma_{ij}^2=1\quad(1\leq i\leq N),
 \qquad \sigstar:=\max_{i,j}\sigma_{ij}.
\end{equation}

\item\label{ass:deformation} \textbf{Finite-rank deformation.} For a fixed integer $r\geq1$ and all $N\geq r$,
\[
 A_N=V_N\Theta V_N^*,\qquad
 V_N^*V_N=I_r,
\]
where $V_N$ is a deterministic $N \times r$ real (resp.\ complex)
matrix and $\Theta$ is a fixed real diagonal $r \times r$ matrix. The columns of $V_N$ may depend on $N$ and have arbitrary support.
\end{enumerate}
\end{definition}

To emphasize the dependence on 
$N$, we shall  sometimes write 
  $W_N=(W_{ij}^{(N)})$. In addition, we impose one of the following two conditions on the entries in \eqref{ass:noise} and the profile in \eqref{ass:profile}. 

\begin{definition}\label{def:entry-regimes}
\leavevmode
\begin{enumerate}[label=\textbf{(FM$_p$)},ref={FM$_p$},leftmargin=4.7em]
\item\label{ass:FM} For some $p\geq4$, the $p$th powers of the entries are uniformly integrable:
\begin{equation}\label{eq:FM-UI}
 \lim_{R\to\infty}\sup_{N\geq1}\sup_{1\leq i\leq j\leq N}
 \E\!\left[|W_{ij}^{(N)}|^p\1_{\{|W_{ij}^{(N)}|>R\}}\right]=0,
\end{equation}
and
\begin{equation}\label{eq:FM-scale}
 \sup_N N(\sigstar)^{p-2}<\infty.
\end{equation}
\end{enumerate}
\begin{enumerate}[label=\textbf{(SW$_\alpha$)},ref={SW$_\alpha$},leftmargin=4.7em]
\item\label{ass:SW} For $\alpha>0$ and a real or complex random variable $Z$, set
\begin{equation}\label{eq:psi-alpha}
 \norm{Z}_{\psi_\alpha}:=
 \inf\bigl\{s>0:\E\big[\exp(|Z|^\alpha/s^\alpha) \big]\leq2\bigr\}.
\end{equation}
For some $K<\infty$ and $0<\alpha\leq2$,
\begin{equation}\label{eq:SW-norm}
 \sup_{N\geq1}\sup_{1\leq i\leq j\leq N}
 \norm{W_{ij}^{(N)}}_{\psi_\alpha}\leq K,
\end{equation}
and
\begin{equation}\label{eq:SW-scale}
 \sigstar(\log N)^{1/\alpha}\longrightarrow0.
\end{equation}
\end{enumerate}
\end{definition}

List the positive and negative eigenvalues of $\Theta$, with multiplicities, as
\begin{equation}\label{eq:spike-order}
 a_1\geq\cdots\geq a_{r_+}>0,
 \qquad b_1\leq\cdots\leq b_{r_-}<0.
\end{equation}
These eigenvalues are fixed, independent of $N$.
Our central result is the following BBP transition at the level of the law of large numbers. 
\begin{theorem}[BBP law of large numbers]\label{thm:main}
Let $X_N$ be as in \cref{def:profile}, and assume either (\ref{ass:FM}) or (\ref{ass:SW}). For each fixed integer $j\geq1$,
\begin{equation}\label{eq:positive-limit}
 \lambda_j(X_N)\longrightarrow
 \begin{cases}
  a_j+a_j^{-1},&j\leq r_+\text{ and }a_j>1,\\
  2,&\text{otherwise},
 \end{cases}
\end{equation}
and
\begin{equation}\label{eq:negative-limit}
 \lambda_{N+1-j}(X_N)\longrightarrow
 \begin{cases}
  b_j+b_j^{-1},&j\leq r_-\text{ and }b_j<-1,\\
  -2,&\text{otherwise}.
 \end{cases}
\end{equation}
All these limits hold in probability under \eqref{ass:FM}
and almost surely under \eqref{ass:SW}.
\end{theorem}
\begin{remark}
In \cref{ex:FM-sharpness}\textup{(\ref{ex:FM-sharpness-as})}, a heavy-tailed Wigner matrix satisfies (\ref{ass:FM}) with $p=4$, while the norm limit fails almost surely. Moreover, the conditions on the sparsity are sharp: \cref{ex:SW-sharpness} shows that the limit in \eqref{eq:SW-scale} cannot in general be relaxed, while \cref{ex:FM-sharpness}\textup{(\ref{ex:FM-sharpness-scale})} shows that the bound in \eqref{eq:FM-scale} cannot in general be replaced by a diverging bound, even one growing arbitrarily slowly. In the finite-moment construction, the entry distributions depend
on $N$. For the sub-Gaussian case, a related construction is given in~\cite[Theorem~1.3]{ASTY}.
\end{remark}

\subsection{Applications}\label{sec:applications}

The model in \cref{def:profile} includes a variety of interesting examples from random matrix theory; here we list only several choices of the variance profile. The spike directions are allowed to be arbitrary and deterministic, subject to (\ref{ass:deformation}).

\begin{example}[Wigner matrices]
For $\sigma_{ij} = N^{-1/2}$, we have $\sigstar = N^{-1/2}$.
In the finite-moment case with $p=4$, condition (\ref{ass:FM}) reduces to uniform
integrability of the fourth powers of the entries, and in the sub-Weibull case \eqref{eq:SW-scale} holds automatically. Thus \cref{thm:main} gives the classical outlier
locations in the respective modes of convergence;
see~\cite{CDF,knowles2013isotropic,renfrew2013finite}.
\end{example}

\begin{example}[Random band matrices]
Index the rows and columns by $(\mathbb{Z}/L\mathbb{Z})^d$, where
$d$ is fixed and $N = L^d$. Let $f$ be a continuous, even,
compactly supported probability density on $\mathbb{R}^d$.
For a bandwidth $w_N \to \infty$ with $w_N/L \to 0$, set
\[
  \sigma_{xy}^2 = \frac{1}{M_N}
  \sum_{k \in \mathbb{Z}^d} f\!\left(\frac{x - y + kL}{w_N}\right),
  \qquad
  M_N = \sum_{z \in \mathbb{Z}^d} f(z/w_N).
\]
The normalization gives stochastic row sums, and Riemann sum
approximation gives $M_N = (1 + o(1))w_N^d$ and
$(\sigstar)^2 = O(\|f\|_\infty w_N^{-d})$. Subject to the respective entry assumptions, \cref{thm:main}
holds if the bandwidth satisfies
\[
 Nw_N^{-d(p-2)/2}=O(1)
 \quad\text{under (\ref{ass:FM})},
 \qquad
 \frac{w_N^d}{(\log N)^{2/\alpha}}\longrightarrow\infty
 \quad\text{under (\ref{ass:SW})}.
\]
For one-dimensional band matrices with uniformly bounded moments of every fixed order, see Au~\cite{Au}.
\end{example}

\begin{example}[Regular support patterns]
Let $Q_N$ be a deterministic symmetric $\{0,1\}$-matrix with exactly
$\kappa_N \geq 1$ ones in each row, and set
$\Sigma_N = \kappa_N^{-1/2} Q_N$. This includes independent weights on the edges of a deterministic regular graph, as in \cite{Au, han2024outliers,ASTY}. Since $\sigstar=\kappa_N^{-1/2}$, the profile conditions become
\[
  N\kappa_N^{-(p-2)/2} = O(1)
  \quad \text{under \eqref{ass:FM}},
  \qquad
  \frac{\kappa_N}{(\log N)^{2/\alpha}} \longrightarrow \infty
  \quad \text{under \eqref{ass:SW}}.
\]
\end{example}
\subsection{Method and structure}\label{sec:method}
Our long preprint \cite{GLZ}, treating both the law of large numbers  in the sub-Gaussian setting and the fluctuations of outlier eigenvalues in the Gaussian setting, has now  been split into two independent  papers. This paper concerns the law of large numbers, with some sections largely revised for brevity and sharpness.
Our argument retains the basic strategy of our previous work~\cite{GLZ}:  deducing
the BBP limits from sharp operator norm bounds on the undeformed
matrix and a fixed-degree polynomial isotropic law,
the latter following~\cite[Proposition~3.6]{Au}. We simplify both the proofs, and  also extend the matrix entries from the sub-Gaussian setting of \cite{GLZ} to sub-Weibull tails with exponent \(\alpha \in (0, 2]\) and to the finite-moment case.

For the spectral norm, \cite{GLZ} obtains quantitative bounds through ribbon graph expansions and comparison with the deformed GOE. Here we combine truncation with the bounded-entry moment estimate of Lata{\l}a, van Handel and Youssef~\cite[Theorem~4.8]{LVHY}. In the sub-Weibull regime, we also adapt the dimension-compression argument in \cite[Remark~2.2]{ASTY}, based on the comparison of walk shapes in \cite[Proposition~2.1]{BvH}. For the polynomial isotropic law, \cite{GLZ} proves concentration by comparing high central moments for sub-Gaussian entries with their Gaussian counterparts through ribbon graph expansions. Truncation extends this argument to sub-exponential entries, but the comparison still requires symmetric entry distributions. Here the entry-resampling moment inequality in \cite[Theorem~2]{BBLM} replaces the high-order graph count without requiring distributional symmetry.

Once these inputs are established, we locate the separated eigenvalues through a finite-dimensional resolvent equation. Both the proofs in \cite[Section~3.2]{Au} and \cite[Sections~4.1--4.3]{GLZ} proceed through spectral measures in spike directions and interlacing. Here polynomial approximation gives convergence of the resolvent compressed to the spike subspace. The classical finite-rank determinant identity and analytic zero counting then identify the outlier locations and multiplicities~\cite[Proposition~5.1 and Section~6]{BGN}. Unlike interlacing, the determinant and zero-counting steps do not require self-adjointness and apply in non-Hermitian settings once the corresponding spectral confinement and isotropic law are established.

For fixed-rank deformations, the present argument gives convergence in probability for entries with uniformly integrable $p$th powers for some $p\geq4$ under $N(\sigma_N^*)^{p-2}=O(1)$. It also extends the sub-Gaussian setting of \cite{GLZ} to uniform sub-Weibull tails with exponent $0<\alpha\leq2$, yielding almost-sure convergence under $\sigma_N^*(\log N)^{1/\alpha}\to0$. In both regimes, we remove the distributional symmetry assumption and the coordinate-support restriction on the deformation, allowing arbitrary deterministic spike directions. The method can also yield quantitative convergence rates or accommodate a rank $r$ that grows sufficiently slowly with $N$, but we do not pursue these extensions here.

The remainder of this paper is organized as follows.
Section~\ref{sec:resolvent} establishes the finite-rank reduction.
Sections~\ref{sec:edge} and~\ref{sec:isotropic} prove the operator
norm bound and the polynomial isotropic law, respectively.
The main theorem is then deduced by applying the reduction in the
appropriate mode of convergence.
Appendix~\ref{sec:finite-order} provides counterexamples
concerning convergence modes and the sharpness of the conditions on the sparsity.

\paragraph{Use of generative AI.}
The framework and idea of this paper are based on previous work \cite{GLZ}. Generative-AI tools were used for technical development and editorial preparation. The authors take responsibility for all mathematical statements and references.
\paragraph{Notation} Throughout this paper, all asymptotics are understood as \(N\to\infty\). All constants \(C,c>0\) are allowed to change from line to line, and their dependence is indicated whenever it is relevant.
\section{Resolvent identity and finite-rank reduction}\label{sec:resolvent}

For every $\musc$-integrable function $f$, write
\[
  \musc(f):=\int f(x)\,\musc(\dd x).
\]
The scope of this section is to deduce  \eqref{eq:positive-limit}--\eqref{eq:negative-limit}
from the following limits, all in the same mode of convergence:
\begin{align}
&\mu_{H_N}\Rightarrow\musc,
 \qquad \norm{H_N}_{\op}\longrightarrow2;
 \label{eq:SC}\\
 &\ip{x_N}{f(H_N)y_N}
   -\ip{x_N}{y_N}\musc(f)\longrightarrow0,
 \label{eq:IGL}
\end{align}
for every fixed polynomial $f$ and every pair of deterministic unit-vector sequences $\{x_N\},\{y_N\}$. %

Set
\begin{equation}\label{eq:resolvent}
 R_N(z):=(zI_N-H_N)^{-1},
 \qquad
 \msc(z):=\int\frac1{z-x}\,\musc(\dd x)
 =\frac{z-\sqrt{z^2-4}}2.
\end{equation}
Here $R_N$ is defined on $\C\setminus\operatorname{spec}(H_N)$, and $\msc$ is defined on $\C\setminus[-2,2]$, with $\sqrt{z^2-4}\sim z$ at infinity.
In the following convergence statements on a compact set outside $[-2,2]$, we set the supremum error to $+\infty$ if the set meets $\operatorname{spec}(H_N)$. Under \eqref{eq:SC}, this event has probability tending to zero;
if \eqref{eq:SC} holds almost surely, it occurs only finitely
often.

We begin with the following lemma, which passes from the polynomial
isotropic limit \eqref{eq:IGL} to the compressed resolvent by a
polynomial approximation argument.

\begin{lemma}\label{lem:compression}
If \eqref{eq:SC}--\eqref{eq:IGL} hold in probability, then for every fixed polynomial $f$,
\begin{equation}\label{eq:polynomial-compression}
 \norm{V_N^*f(H_N)V_N-\musc(f)I_r}_{\op}\xtoP0,
\end{equation}
and for every compact set $\mathcal K\Subset\C\setminus[-2,2]$,
\begin{equation}\label{eq:resolvent-compression}
 \sup_{z\in\mathcal K}
 \norm{V_N^*R_N(z)V_N-\msc(z)I_r}_{\op}\xtoP0.
\end{equation}
If \eqref{eq:SC}--\eqref{eq:IGL} hold almost surely,  both of the above
 hold almost surely.
\end{lemma}

\begin{proof}
We prove the statements for convergence in probability; the almost sure case follows
by the same argument.
Since $r$ is fixed, \eqref{eq:polynomial-compression} follows
directly from \eqref{eq:IGL}.
Fix $\mathcal{K} \Subset \mathbb{C} \setminus [-2,2]$ and $\varepsilon>0$.
By \eqref{eq:SC}, with probability tending to one the spectrum of
$H_N$ is contained in a deterministic compact interval $J$ at
positive distance from $\mathcal{K}$. At each point $z$ of a  finite  $\varepsilon$-net $\mathcal{E}$ in $\mathcal K$, the Weierstrass approximation theorem allows us to approximate
$x \mapsto (z-x)^{-1}$ uniformly on $J$ to within $\varepsilon$
by a polynomial $p_z$. The resolvent identity gives
\[
 \norm{R_N(z)-R_N(w)}_{\op}\leq C_{\mathcal K,J}|z-w|,
\]
and the same bound holds for $\msc$. 
Consequently,
\begin{equation*}
\begin{aligned}
\bigl\|V_N^* R_N(z) V_N - \msc(z) I_r\bigr\|_{\op}&
  \leq \sup_{z \in \mathcal{E}}
  \bigl\|V_N^* R_N(z) V_N - \msc(z) I_r\bigr\|_{\op}
  + C_{\mathcal{K},J} \varepsilon\\
  &\leq 
  \sup_{z \in \mathcal{E}}\bigl\|V_N^* p_z(H_N) V_N - \musc(p_z) I_r\bigr\|_{\op}+C'_{\mathcal{K},J}\varepsilon.
\end{aligned}
\end{equation*}
Applying \eqref{eq:polynomial-compression} to the polynomials $(p_z)_{z \in \mathcal{E}}$ completes the proof.
\end{proof}

The determinant identity below is a standard finite-rank
reduction; see e.g.~\cite[Proposition~5.1]{BGN}
and~\cite[Lemma~2.1 and Remark~2.2]{MR3010398}. The stated mode of convergence follows from \cref{lem:compression}.

\begin{lemma}\label{lem:determinant}
For $z\notin\operatorname{spec}(H_N)$,
\begin{equation}\label{eq:determinant-identity}
 \det(zI_N-X_N)=\det(zI_N-H_N)
 \det\!\left(I_r-\Theta V_N^*R_N(z)V_N\right).
\end{equation}
If \eqref{eq:SC}--\eqref{eq:IGL} hold in probability (respectively, almost surely), then locally uniformly on $\C\setminus[-2,2]$ in probability (respectively, almost surely),
\begin{equation}\label{eq:F-convergence}
 F_N(z):=\det\!\left(I_r-\Theta V_N^*R_N(z)V_N\right)
 \longrightarrow
 F(z):=\det\!\left(I_r-\msc(z)\Theta\right).
\end{equation}
\end{lemma}

\begin{proof} \eqref{eq:determinant-identity} follows from Sylvester's identity $\det(I-BC)=\det(I-CB)$
together with the following well-known identity: 
\[
 zI_N-X_N=(zI_N-H_N)
 \bigl(I_N-R_N(z)V_N\Theta V_N^*\bigr).
\]
The convergence immediately follows from \cref{lem:compression} and the continuity of the determinant.
\end{proof}
For $a\in\R\setminus\{0\}$, write $\rho(a):=a+a^{-1}$. The branch in \eqref{eq:resolvent} satisfies %
\[
 \msc(z)^2-z\msc(z)+1=0,\qquad
 1-a\msc(z)=0\ \Longleftrightarrow\ |a|>1,\ z=\rho(a),\qquad
 \msc'(\rho(a))=\frac1{1-a^2}\neq0.
\]
The middle equivalence is for $z\notin[-2,2]$, and the derivative identity is for $|a|>1$. Since $\Theta$ is diagonal, a repeated spike gives a zero of the same multiplicity in $F$.

We are now ready to prove the main result of this section.
The argument uses the analytic zero-counting method
of~\cite[Section~6.2]{BGN} and~\cite[Remark~2.2]{MR3010398}.

\begin{theorem}[Abstract BBP theorem]\label{thm:abstract-transfer}
If \eqref{eq:SC}--\eqref{eq:IGL} hold in probability (respectively, almost surely), then \eqref{eq:positive-limit}--\eqref{eq:negative-limit} hold in probability (respectively, almost surely). 
\end{theorem}

\begin{proof}
It suffices to consider the almost sure case. The case of convergence in probability then follows by applying the argument to an almost surely convergent subsequence of an arbitrary subsequence and using the subsequence criterion.
By Lemma~\ref{lem:determinant}, $F_N(z)$ converges to $F(z)$
locally uniformly on $\mathbb{C} \setminus [-2,2]$, almost surely.
The zeros of $F$ are precisely
\[
\rho(a)=a+a^{-1},\qquad |a|>1,
\]
with multiplicities equal to those of the corresponding spikes.

Hence, by Rouché's theorem and the identity \eqref{eq:determinant-identity}, the eigenvalues of 
$X_N$
 that remain separated from 
$[-2,2]$
 converge, with the corresponding multiplicities, to these points. Moreover, the bound $\norm{X_N}_{\op}\leq\norm{H_N}_{\op}+\norm{\Theta}_{\op}$ ensures that the spectrum of $X_N$ is eventually contained in a fixed compact interval. On the compact set obtained by removing from this interval an
$\varepsilon$-neighborhood of $[-2,2]$ and of each outlier disc,
$F$ has no zeros and is therefore bounded away from zero. Local uniform convergence then excludes any further zeros of $F_N$, and hence any additional separated eigenvalues of $X_N$. By local uniform convergence, $F_N$
 has no zeros on this set for all sufficiently large 
$N$. Consequently, $X_N$
has no additional eigenvalues separated from $[-2,2]$.

Moreover, since $\operatorname{rank}(X_N - H_N) \leq r$,
\eqref{eq:SC} and the rank inequality~\cite[Theorem~A.43]{BS}
give $\mu_{X_N} \Rightarrow \mu_{\mathrm{sc}}$ almost surely.
After removing the outliers, this convergence forces each fixed upper and lower extreme eigenvalue to converge to $2$ and $-2$,
respectively.%
\end{proof}

\begin{remark}[Rates for separated outliers]\label{rem:outlier-rate}

Fix a spike $a$ with $|a|>1$, and choose a sufficiently small closed interval $J$ around $\rho(a)$ that is disjoint from $[-2,2]$ and contains no other limiting outlier. Whenever $J\cap\operatorname{spec}(H_N)=\varnothing$, every eigenvalue $\lambda\in J$ of $X_N$ satisfies
\[
 |\lambda-\rho(a)|
 \leq C\sup_{x\in J}
 \norm{V_N^*R_N(x)V_N-\msc(x)I_r}_{\op},
\]
where $C$ depends only on the fixed spikes and $J$. Indeed, \eqref{eq:determinant-identity} implies
\[
 \min_{\theta\in\operatorname{spec}(\Theta)}
 |1-\theta\msc(\lambda)|
 \leq\norm{\Theta}_{\op}
 \norm{V_N^*R_N(\lambda)V_N-\msc(\lambda)I_r}_{\op}.
\]
The factor for $a$ has a simple zero at $\rho(a)$, and the factors for the other distinct spike values stay away from zero on $J$. Thus the left side is bounded below by a positive constant times $|\lambda-\rho(a)|$. A quantitative bound therefore gives an outlier bound of the same order, also for repeated spikes.
\end{remark}

\section{Operator norm of the undeformed matrix}\label{sec:edge}

The goal of this section is to prove \eqref{eq:SC} for $H_N$
under \eqref{ass:FM} or \eqref{ass:SW}. 
The upper bound on $\|H_N\|_{\op}$
  is proved by truncation and the method of moments, while the matching lower bound is established via the convergence of the semicircle law.

\subsection{Upper and lower norm bounds}

We begin with the following quantitative bound on the operator norm
for matrices with bounded entries. The real symmetric case of the following bound is due to Lata{\l}a, van Handel and Youssef~\cite[Theorem~4.8]{LVHY}. We prove the complex Hermitian case by comparison with a real symmetric matrix.

\begin{lemma}[Moment bound, {\cite[Theorem~4.8]{LVHY}}]
\label{lem:bounded-Hermitian-moment}
Let $B=(b_{ij})_{i,j\leq N}$ be a real symmetric or complex
Hermitian matrix whose upper-triangular entries are independent, bounded, and centered. Set \(s_{ij}:=\E[|b_{ij}|^2].\) Then, for every integer $q\geq1$,
\begin{align}
\left(\E\big[\Tr B^{2q}\big]\right)^{1/(2q)}
&\leq
2\Big(\sum_i\big(\sum_j s_{ij}\big)^q\Big)^{1/(2q)}
\nonumber+C\sqrt q
\Big(\sum_{i,j}\norm{b_{ij}}_\infty^{2q}\Big)^{1/(2q)},
\label{eq:bounded-Hermitian-moment}
\end{align}
where $C>0$ is a universal constant.
\end{lemma}
\begin{proof}
The main idea is to bound the moments of $B$ in terms of those of a real symmetric matrix $C$ defined below. For $i\leq j$, set $c_{ij}=0$ if $\norm{b_{ij}}_\infty=0$; otherwise, let $c_{ij}$ be a real random variable with
\[
\Pp\!\left(c_{ij}=2\norm{b_{ij}}_\infty\right)=\frac{s_{ij}}{6\norm{b_{ij}}_\infty^2},\qquad \Pp\!\left(c_{ij}=-\norm{b_{ij}}_\infty\right)=\frac{s_{ij}}{3\norm{b_{ij}}_\infty^2},\qquad \Pp(c_{ij}=0)=1-\frac{s_{ij}}{2\norm{b_{ij}}_\infty^2}.
\]
Since $s_{ij}\leq\norm{b_{ij}}_\infty^2$, the distribution of $c_{ij}$ is well-defined, and
\begin{equation}\label{eq:cij}
\E[c_{ij}]=0,\qquad \E[c_{ij}^2]=s_{ij},\qquad \norm{c_{ij}}_\infty\leq2\norm{b_{ij}}_\infty.
\end{equation}
Moreover, for every integer $m\geq2$,
\[
\E[c_{ij}^m]=\frac{s_{ij}\norm{b_{ij}}_\infty^{m-2}}6\bigl(2^m+2(-1)^m\bigr)\geq s_{ij}\norm{b_{ij}}_\infty^{m-2}\geq\E[|b_{ij}|^m].
\]
Choose $\{c_{ij}:i\leq j\}$ independently, set $c_{ji}=c_{ij}$, and write $C=(c_{ij})_{i,j\le N}$. To show that the moments of $B$ are bounded by those of $C$, we expand both moments in terms of closed walks $(i_0i_1\ldots i_{2q-1}i_0)$. We  associate each walk with its connected graph of distinct indices and unordered traversed pairs, allowing loops. Let $k_e$ denote the multiplicity of the edge $e$. By independence, any walk with an edge of multiplicity one contributes zero. Hence, it suffices to consider closed walks satisfying $k_e\geq 2$ for every edge $e$. By independence and Hermitian symmetry,

\[
\Big|\E\Big[\prod_{\ell=1}^{2q}b_{i_{\ell-1}i_\ell}\Big]\Big|\leq\prod_e\E\big[|b_e|^{k_e}\big]\leq\prod_e\E\big[c_e^{k_e}\big]=\E\Big[\prod_{\ell=1}^{2q}c_{i_{\ell-1}i_\ell}\Big].
\]
Summing over all closed walks yields
$\E[\Tr B^{2q}] \leq \E[\Tr C^{2q}].$
Applying \cite[Theorem~4.8]{LVHY} to $C$, together with
\eqref{eq:cij}, completes the proof.
\end{proof}
Lemma~\ref{lem:bounded-Hermitian-moment} yields the following norm bound.

\begin{lemma}\label{lem:bounded-edge}
Let $B_N=(b_{ij})_{i,j\leq N}$ be real symmetric or complex Hermitian with independent centered upper-triangular entries. Put $s_{ij}=\E[|b_{ij}|^2]$ and suppose
\begin{equation}\label{eq:bounded-edge-assumptions}
 \max_i\sum_j s_{ij}\leq1+o(1),
 \qquad |b_{ij}|\leq d_N\quad\text{a.s.},
 \qquad d_N\sqrt{\log N}\longrightarrow0.
\end{equation}
Then
\begin{equation}\label{eq:bounded-edge-upper}
 \limsup_{N\to\infty}\norm{B_N}_{\op}\leq2
 \qquad\text{almost surely}.
\end{equation}
If, in addition, $\mu_{B_N}\Rightarrow\musc$ in probability, then $\norm{B_N}_{\op}\xtoP 2$. If the empirical spectral distributions converge almost surely, then $\norm{B_N}_{\op}\to2$ almost surely.
\end{lemma}

\begin{proof}
The case $d_N=0$ is trivial, so we assume henceforth that $d_N>0$ and set
\[
q_N=\Big\lfloor\frac{\sqrt{\log(eN)}}{d_N}\Big\rfloor.\]
Thus, \(q_N/\log N\to\infty\) and \(d_N\sqrt{q_N}\to0\). By \Cref{lem:bounded-Hermitian-moment},
\[
\left(\mathbb{E}\big[\operatorname{Tr} B_N^{2q_N}\big]\right)^{1/(2q_N)}
\le 2+o(1),
\]
where we have used \(N^{1/(2q_N)},N^{1/q_N}\to1\). Since $\|B_N\|_{\mathrm{op}}^{2q_N}\leq\operatorname{Tr} B_N^{2q_N}$, we apply Markov's inequality to obtain, for every fixed $\varepsilon>0$ and all sufficiently large $N$,
\[
 \mathbb{P}(\|B_N\|_{\mathrm{op}}>2+\varepsilon) \leq \frac{\mathbb{E}\big[\operatorname{Tr} B_N^{2q_N}\big]}{(2+\epsilon)^{2q_N}}
 \leq\left(\frac{2+\varepsilon/2}{2+\varepsilon}\right)^{2q_N}.
\]
This bound is summable in $N$, since $q_N/\log N\to\infty$. We conclude by the Borel--Cantelli lemma (together with a countable intersection over $\varepsilon\downarrow0$) that \eqref{eq:bounded-edge-upper} holds. The matching lower bound follows from \cref{lem:as-semicircle}.
\end{proof}
The next lemma follows from a semicircle criterion~\cite{Chin} and implies the desired lower bound on the operator norm.

\begin{lemma}[Almost-sure semicircle law]\label{lem:as-semicircle}
Under either (\ref{ass:FM}) or (\ref{ass:SW}),
\begin{equation}\label{eq:as-semicircle}
 \mu_{H_N}\Rightarrow\musc\qquad\text{almost surely}.
\end{equation}
\end{lemma}

\begin{proof}
The averaged row-variance condition holds because
\[
 \frac1N\sum_{i=1}^N
 \Big|\sum_{j=1}^N\E\big[|H_{ij}|^2\big]-1\Big|
 \leq(\sigstar)^2=o(1);
\]
the left side is zero in the complex case. For every $\varepsilon>0$,
\[
 \frac1N\sum_{i,j}\E\!\left[|H_{ij}|^2
 \1_{\{|H_{ij}|>\varepsilon\}}\right]
 \leq C\sup_{1\leq k\leq\ell\leq N}
 \E\Big[|W_{k\ell}^{(N)}|^2
 \1_{\{|W_{k\ell}^{(N)}|>\varepsilon/\sigstar\}}\Big]=o(1).
\]
The last limit follows from \eqref{eq:FM-UI} or \eqref{eq:SW-norm}. Therefore, the hypotheses of \cite[Theorem~1.6 and Remark~1.7]{Chin} are satisfied, which implies that $\E[\mu_{H_N}]\Rightarrow\musc$; \cite[Lemma~1.3]{Chin} then gives almost-sure convergence.
\end{proof}

\subsection{Norm convergence: finite moments}
This section establishes the convergence of the spectral norm of $H_N$
under \eqref{ass:FM}. We proceed by truncation, beginning with
the following lemma, which shows that the spectral norm of the discarded tail matrix tends to zero in probability.
 
\begin{lemma}\label{lem:tail-matching}
Assume (\ref{ass:FM}) and set
\[
 \tau_N:=\frac1{\log(eN)},
 \qquad T_N:=\bigl(H_{ij}\1_{\{|H_{ij}|>\tau_N\}}\bigr)_{i,j\le N},
 \qquad M_N:=\E[T_N].
\]
Then
\begin{equation}\label{eq:tail-norm-zero}
 \norm{T_N-M_N}_{\op}\xtoP0.
\end{equation}
\end{lemma}

\begin{proof}
For each $i$, let $q_i:=\sum_j\Pp(T_{ij}\ne0)$ be the expected number of nonzero entries in row $i$ of $T_N$. By \eqref{eq:FM-UI}, the $p$-th moments of $W_{ij}^{(M)}$ are uniformly bounded. Thus,
\begin{equation*}
 q_i =\sum_j\Pp(|H_{ij}| \ge \tau_N) \le \sum_{j}\tau_N^{-p}\mathbb{E}[|H_{ij}|^p]\leq C \tau_N^{-p}\sum_j\sigma_{ij}^p
 \leq C \tau_N^{-p}(\sigstar)^{p-2}\sum_j\sigma_{ij}^2
 =C \tau_N^{-p}(\sigstar)^{p-2}.
\end{equation*}
By independence of the entries within each row, %

\begin{equation*}
 \Pp\bigl(\exists i:\#\{j:T_{ij}\ne0\}\geq2\bigr) 
\leq \sum_{i}\sum_{j_1 \neq j_2}\Pp(T_{ij_1} \neq 0) \Pp (T_{ij_2} \neq 0) \le \sum_i q_i^2
 \leq CN\tau_N^{-2p}(\sigstar)^{2p-4}=o(1),
\end{equation*}
 where in the last equation we use \eqref{eq:FM-scale}. Thus, with probability tending to one, %
 each row of $T_N$ has at most one nonzero entry. On this event, $T_NT_N^*$ is diagonal and thus $\norm{T_N}_{\op}=\max_{i,j}|(T_N)_{ij}| \le \max_{i,j}|H_{ij}|$. 

For every fixed $\varepsilon>0$,
\begin{align*}
\Pp\Big(\max_{i,j}|H_{ij}|>\varepsilon\Big)&\leq \sum_{i,j}\varepsilon^{-p}\E\big[|H_{ij}|^p1_{\{H_{ij}| \ge \varepsilon\}}\big]\\
&\leq \sum_{i,j}\varepsilon^{-p}\sigma_{ij}^p\E\Big[|W_{ij}|^p1_{\{W_{ij}| \ge \varepsilon/\sigstar\}}\Big] \\
 &\leq \varepsilon^{-p}N(\sigstar)^{p-2}
\sup_{N\geq1}\sup_{1\leq k\leq\ell\leq N}
\E\Big[|W_{k\ell}^{(N)}|^p
 \1_{\{|W_{k\ell}^{(N)}|>\varepsilon/\sigstar\}}\Big]=o(1),
\end{align*}
where the third inequality uses the doubly stochastic property on
$(\sigma_{ij}^2)_{i,j\leq N}$, and the last follows from~\eqref{ass:FM}. Hence $\norm{T_N}_{\op}\to0$ in probability.

The  matrix $M_N$ satisfies
\begin{align*}
 \norm{M_N}_{\op}
 \leq \max_i\sum_j|M_{ij}|
 \leq C \tau_N^{1-p}\max_i\sum_j\sigma_{ij}^p 
 \leq C \tau_N^{1-p}(\sigstar)^{p-2}
 \leq \frac{C(\log(eN))^{p-1}}{N}=o(1),
\end{align*}
where the last inequality follows from~\eqref{eq:FM-scale} again. The conclusion then follows by the triangle inequality.
\end{proof}

Combining \cref{lem:tail-matching} with \cref{lem:bounded-edge,lem:as-semicircle} yields the norm limit in the finite-moment case.

\begin{proposition}\label{prop:finite-edge}
Under (\ref{ass:FM}),
\begin{equation}\label{eq:finite-edge}
 \norm{H_N}_{\op}\xtoP2.
\end{equation}
\end{proposition}

\begin{proof}With $T_N$ and $M_N$ as in Lemma \ref{lem:tail-matching},
set
\begin{equation}\label{eq:finite-cutoff}
 B_N:=H_N-T_N+M_N,
 \qquad H_N=B_N+(T_N-M_N).
\end{equation}
The entries of $B_N$ are centered and satisfy $|B_{ij}|\leq2\tau_N$, with $\tau_N\sqrt{\log N}\to0$. Since $H_{ij}-T_{ij}$ and $T_{ij}$ have disjoint supports and $\E[H_{ij}]=0$,
\begin{equation*}
\begin{aligned}
0\leq \E[|H_{ij}|^2]-\E[|B_{ij}|^2]
&=\E[|T_{ij}|^2]+|\E[T_{ij}]|^2\\
&\leq2\sigma_{ij}^2\E\!\left[|W_{ij}|^2
 \1_{\{|W_{ij}|>\tau_N/\sigstar\}}\right]
 =:2\sigma_{ij}^2\epsilon_{ij}.
\end{aligned}
\end{equation*}
Hence,   
\begin{align}
\max_i\Big|\sum_j\E[|B_{ij}|^2]-1\Big|
& \leq (\sigstar)^2+2\max_i\sum_j\sigma_{ij}^2\epsilon_{ij} \notag\\
 &\leq (\sigstar)^2+2\max_{1\leq k\leq\ell\leq N}\epsilon_{k\ell}=o(1),
 \label{eq:finite-row-defect}
 \end{align}

where in the last step we use \eqref{eq:FM-UI} and $\tau_N/\sigstar \to \infty$. The term $(\sigstar)^2$ accounts for the real diagonal convention.
The upper bound then follows from \cref{lem:bounded-edge,lem:tail-matching},
and the lower bound from \cref{lem:as-semicircle}.
\end{proof}
\subsection{Norm convergence: sub-Weibull entries}
This section establishes the convergence of the spectral norm of
$H_N$ under \eqref{ass:SW}. Recalling \eqref{eq:psi-alpha}, the sub-Weibull tail estimate
of~\cite[Theorem~2.1]{VGNA} gives, under \eqref{ass:SW},
\begin{equation}\label{eq:SW-tail}
\Pp(|W_{ij}|>u)\leq2e^{-(u/K)^\alpha},\qquad u\geq0.
\end{equation}

The proof of the following bound relies on a truncation-and-compression argument, following \cite[Remark~2.2]{ASTY}. The compression uses the walk-shape comparison of~\cite[Proposition~2.1]{BvH}.

\begin{lemma}[Sub-Weibull trace bound]\label{lem:ASTY}
For a real symmetric or complex Hermitian matrix defined by
\[
H_N^\circ=(\sigma_{ij}W_{ij}\1_{\{i\ne j\}})_{i,j \le N},
\]
Suppose that $\Sigma_N$ satisfies \eqref{eq:stochastic}, and that the off-diagonal upper-triangular entries $\{W_{ij}:i<j\}$ are independent and centered, with
\[
\E\big[|W_{ij}|^2\big]=1,\qquad \norm{W_{ij}}_{\psi_\alpha}\leq K,\qquad 0<\alpha\leq2.
\]
Then, for every integer $k\geq\lceil\log(eN)\rceil$ and every $L\geq1$,
\begin{equation}\label{eq:ASTY-bound}
\left(\E\big[\Tr(H_N^\circ)^{2k}\big]\right)^{1/(2k)}
\leq N^{1/k}\big(2+Ce^{-cL^\alpha}+C\sigstar Lk^{1/\alpha}\big),
\end{equation}
where $C,c>0$ depend only on $K$ and $\alpha$.
\end{lemma}

\begin{proof}
For $i<j$, set
\[
U_{ij}:=\sigma_{ij}\Big(W_{ij}\1_{\{|W_{ij}|\leq L\}}-\E[W_{ij}\1_{\{|W_{ij}|\leq L\}}]\Big),\qquad
T_{ij}:=\sigma_{ij}W_{ij}\1_{\{|W_{ij}|>L\}},
\]
and complete $U_N,T_N$ by Hermitian symmetry with zero diagonal entries. Since
\[
\max_i\sum_j\E\big[|U_{ij}|^2\big]\leq\max_i\sum_{j\neq i}\sigma_{ij}^2\le1\quad  \text{and} \quad \norm{U_{ij}}_\infty\leq2\sigstar L,
\]
it follows from \cref{lem:bounded-Hermitian-moment} that
\begin{equation}\label{eq:ASTY-bounded-part}
\left(\E\big[\Tr U_N^{2k}\big]\right)^{1/(2k)}
\leq N^{1/k}\left(2+C\sigstar L\sqrt{k}\right).
\end{equation}
Set $C_N:=T_N-\E[T_N]$, so that $H_N^\circ=U_N+C_N$. Let $T_N'$ be an independent copy of $T_N$, and let $\mathcal R_N$ be an independent real symmetric Rademacher matrix with zero diagonal. By Jensen's inequality, symmetrization~\cite[Section~6.1]{LT}, and the Schatten triangle inequality,
\begin{align}
\left(\E\big[\Tr C_N^{2k}\big]\right)^{1/(2k)}
&\leq\left(\E\big[\Tr(T_N-T_N')^{2k}\big]\right)^{1/(2k)}\leq2\left(\E\big[\Tr(\mathcal R_N\circ T_N)^{2k}\big]\right)^{1/(2k)}.
\label{eq:ASTY-tail-symmetrization}
\end{align}
Indeed, the upper-triangular entries of $T_N-T_N'$ are symmetric, so multiplication by independent Rademacher signs preserves their joint law. 

 We now bound the moments of $\mathcal{R}_N\circ T_N$. Let $\zeta_\alpha$ be a symmetric real-valued Weibull random variable normalized so that $\E[|\zeta_\alpha|^2]=1$. Then $\norm{\zeta_\alpha}_{L^m}\asymp m^{1/\alpha}$, and the sub-Weibull tail and moment bounds (see, for instance, \cite[Theorem~2.1]{VGNA}) imply, uniformly in $i,j$ and $m\geq2$,
\begin{equation}\label{equ:moment_1_SW}
\E\left[|W_{ij}|^m\1_{\{|W_{ij}|>L\}}\right]
\leq\left(\E|W_{ij}|^{2m}\right)^{1/2}\Pp(|W_{ij}|>L)^{1/2}
\leq C^m e^{-cL^\alpha}\E[|\zeta_\alpha|^m].
\end{equation}
In the Hermitian case, note that
\[
\left|\E\left[W_{ij}^a\overline{W_{ij}}^{\,b}\1_{\{|W_{ij}|>L\}}\right]\right|
\leq\E\left[|W_{ij}|^{a+b}\1_{\{|W_{ij}|>L\}}\right].
\]
Thus, the moments of $W_{ij}$ can likewise be bounded in terms of $\xi_{\alpha}$.

Expand the $2k$ moments into closed walks $\boldsymbol u=(u_1u_2\dots u_{2k}u_1)$. Since the entries of $\mathcal{R}_N \circ T_N$ symmetrically
distributed, only closed walks in which each distinct edge is
traversed an even number of times contribute to the expansion. Group the surviving closed walks according to their shapes, as in~\cite[Proposition~2.1 and Lemmas~2.5--2.6]{BvH}. %
Fix a shape $s$ with $m(s)$ vertices and edge multiplicities $(k_e)_e$.  Since $s$ is connected, it has at least $m(s)-1$
distinct edges, and
\[
\sum_{\boldsymbol u:\,s(\boldsymbol u)=s}
\prod_{a=1}^{2k}\sigma_{u_a u_{a+1}}
\leq N(\sigstar)^{2k-2(m(s)-1)}.
\]
Indeed, retain two profile factors on each edge of a spanning tree and bound all remaining factors by $\sigstar$. Summing the tree factors by successive leaf removal and using $\sum_j\sigma_{ij}^2=1$ then yields the above bound. Combining with~\eqref{equ:moment_1_SW}, the contribution of
shape $s$ to the moment of $\mathcal{R}_N \circ T_N$ is at most
\[
C^{2k}N(\sigstar)^{2k-2(m(s)-1)}
e^{-c(m(s)-1)L^\alpha}
\prod_e\E\big[|\zeta_\alpha|^{k_e}\big].
\]
On the other hand, set
\[
  n := \Bigl\lceil (\sigstar)^{-2} e^{-cL^\alpha} \Bigr\rceil + k,
\]
and let $G_n$ be the zero-diagonal real symmetric matrix with i.i.d.\
off-diagonal entries distributed as $\zeta_\alpha$.  
Applying the moment expansion  again to $G_n$, we obtain that the contribution of the same shape $s$ with $m(s)$ vertices and edge multiplicities $(k_e)_{e}$ is
 
\[\frac{n!}{(n-m(s))!}\prod_e\E\big[|\zeta_\alpha|^{k_e}\big].\]
Since $m(s)\leq k+1$,
\[
\frac{n!}{(n-m(s))!}
\geq n\left[(\sigstar)^{-2}e^{-cL^\alpha}\right]^{m(s)-1}.
\]
Thus, for each shape $s$, its contribution to the moments of
$\mathcal{R}_N \circ T_N$ can be bounded by its contribution to
those of $G_n$. Summing over all shapes yields
\[
\E\big[\Tr(\mathcal R_N\circ T_N)^{2k}\big]
\leq C^{2k}(\sigstar)^{2k}\frac{N}{n}\E\big[\Tr G_n^{2k}\big].
\]
We now apply the moment bound from \cite[Remark~2.2]{ASTY} with $L=1$ to $(n-1)^{-1/2}G_n$ to obtain
\[
\left(\E[\Tr G_n^{2k}]\right)^{1/(2k)}
\leq C_\alpha n^{1/k}\left(\sqrt n+k^{1/\alpha}\right).
\]
Since \eqref{eq:stochastic} implies $\sigstar\geq N^{-1/2}$, we have $n\leq N+k+1$. Hence, for $k\geq\log(eN)$, 
\[
n^{1/(2k)}\leq C,\qquad
\sigstar\sqrt n\leq C\left(e^{-cL^\alpha/2}+\sigstar\sqrt k\right).
\]
Taking the $2k$-th root and using $\sqrt{k} \leq k^{1/\alpha}$ gives
\begin{equation}\label{eq:ASTY-tail-part}
\left(\E[\Tr(\mathcal R_N\circ T_N)^{2k}]\right)^{1/(2k)}
\leq CN^{1/(2k)}
\left(e^{-cL^\alpha}+\sigstar k^{1/\alpha}\right).
\end{equation}
Since $\alpha \leq 2$, $L \geq 1$, and $N^{1/(2k)} \leq N^{1/k}$,
the bound~\eqref{eq:ASTY-bound} follows by combining
\eqref{eq:ASTY-bounded-part}, \eqref{eq:ASTY-tail-symmetrization},
and~\eqref{eq:ASTY-tail-part} with the triangle inequality.
\end{proof}

The following consequence extends the Weibull norm limit established in \cite[Remark~2.2]{ASTY} to random matrices whose entries satisfy assumption 
(\ref{ass:SW}).

\begin{proposition}\label{prop:SW-edge}
Under (\ref{ass:SW}),
\begin{equation}\label{eq:SW-edge}
\norm{H_N}_{\op}\longrightarrow2\qquad\text{almost surely}.
\end{equation}
\end{proposition}

\begin{proof}

Set $D_N:=\operatorname{diag}(H_{11},\ldots,H_{NN}).$
For every fixed $\varepsilon>0$, \eqref{eq:SW-tail} yields
\begin{equation}\label{eq:diag-SW}
\Pp(\norm{D_N}_{\op}>\varepsilon)
\leq CNe^{-c\,(\varepsilon/\sigstar)^{\alpha}}.
\end{equation}
The upper bound is summable in $N$ by \eqref{eq:SW-scale}. Hence, by the Borel--Cantelli lemma, $\norm{D_N}_{\op}\to0$ a.s..

Put $c_N:=\sigstar(\log N)^{1/\alpha}\to0$. With $c$ as in \eqref{eq:ASTY-bound}, choose
\[
k_N:=\left\lfloor (c_N)^{-\alpha/2} \, {\log N} \right\rfloor,\qquad
L_N^\alpha:=c^{-1}\log\frac1{\sigstar k_N^{1/\alpha}}.
\]
Then, for all sufficiently large $N$, we have $k_N\geq\lceil\log(eN)\rceil$, $L_N\geq1$, and
\[
\frac{k_N}{\log N}\to\infty,\qquad
N^{1/k_N}\to1,\qquad
\sigstar k_N^{1/\alpha}=(1+o(1))\sqrt{c_N}.
\]
Moreover,
\[
e^{-cL_N^\alpha}=\sigstar k_N^{1/\alpha},\qquad
\sigstar L_Nk_N^{1/\alpha}
=O\left(\sqrt{c_N}\log^{1/\alpha}\frac1{c_N}\right)=o(1).
\]
Applying \cref{lem:ASTY} to $H_N^\circ=H_N-D_N$ therefore gives
\begin{equation}\label{eq:ASTY-2plus}
\left(\E\big[\Tr(H_N-D_N)^{2k_N}\big]\right)^{1/(2k_N)}
\leq2+o(1).
\end{equation}
Hence, for every $\eta>0$,
\[
\Pp\bigl(\norm{H_N-D_N}_{\op}>2+\eta\bigr)
\leq\left(\frac{2+o(1)}{2+\eta}\right)^{2k_N}.
\]
This bound is summable since $k_N/\log N\to\infty$. The Borel--Cantelli lemma, together with a countable intersection over $\eta \downarrow 0$, yields 
\[
\limsup_{N\to\infty}\norm{H_N-D_N}_{\op}\leq2
\qquad\text{almost surely}.
\]
Since $\norm{D_N}_{\op}\to0$ almost surely, we conclude that \[
\limsup_{N\to\infty}\norm{H_N}_{\op}\leq2
\qquad\text{almost surely}.
\]%
The matching lower bound follows from  \cref{lem:as-semicircle}.
\end{proof}

\section{Polynomial isotropic laws}\label{sec:isotropic}

In this section, we prove the polynomial isotropic law stated in \eqref{eq:IGL}. We first treat the case
of bounded entries, then pass to the two entry regimes by truncation.

\subsection{Expectation and concentration for bounded cores}

Write $m_t^{\rm sc}:=\musc(x^t)$ for $t\in \mathbb{N}$. The next lemma adapts the argument of~\cite[Lemma~3.1]{Au} to bounded entries with approximately normalized variance sums.

\begin{lemma}\label{lem:bounded-core-bias}
Let $B=(b_{ij})_{i,j\leq N}$ be a real symmetric or complex Hermitian random matrix whose upper-triangular entries are independent and centered. Set
\[
s_{ij}:=\E[|b_{ij}|^2],\qquad
 \varepsilon:=\max_i\Big|\sum_j s_{ij}-1\Big|.
\]
Assume $\varepsilon\leq1$ and $|b_{ij}|\leq d\leq1$ almost surely. For every fixed integer $t\geq1$,
\begin{equation}\label{eq:bounded-bias}
 \norm{\E[B^t]-m_t^{\rm sc}I_N}_{\op}
 \leq C_t(\varepsilon+d).
\end{equation}
\end{lemma}

\begin{proof}
\emph{Step 1: Walk expansion.} Fix deterministic unit vectors $x, y \in \mathbb{C}^{N}$.
Expanding $\langle x, \mathbb{E}[B^{t}]y \rangle$ yields
\[
  \langle x, \mathbb{E}[B^{t}]y \rangle
  = \sum_{i_{0},\ldots,i_{t}}
    \overline{x_{i_{0}}}\,y_{i_{t}}\,
    \mathbb{E}\!\left[b_{i_{0}i_{1}}\cdots b_{i_{t-1}i_{t}}\right].
\]
We group the terms according to their \emph{shape}, following the
framework of~\cite[Proposition~2.1 and Lemmas~2.5--2.6]{BvH}.
Associate with each walk its connected graph $G=(V,E)$,
whose vertices are the distinct indices and whose edges are the
distinct unordered pairs traversed; loops are permitted. Relabel the vertices $V$ in order of
first appearance, and let $ \alpha \in \{1,\ldots,|V|\}$
be the label of the end vertex $i_t$. Write $k_{e}$ for the multiplicity of edge $e$. By independence, a walk contributes to the expectation
only if $k_{e} \geq 2$ for each $e$.

\emph{Step 2: Leading shapes.} 
A shape is \emph{leading} if $G$ is a tree and every edge is
traversed exactly twice.
Such walks are necessarily closed ($i_0 = i_t$, i.e.\ $\alpha=1$).
The leading shapes of length \(t=2k\) are in bijection with
Dyck paths of length \(2k\), and their number is the Catalan number $m_{2k}^{\rm sc}$. There are no leading shapes when $t$ is odd.
 
Fix a leading shape $\tau$ with $t/2$ edges.
Each edge traversed twice contributes a moment factor
$\mathbb{E}[|b_{ij}|^{2}] = s_{ij}$.
We sum over all labelings of $\tau$ by successive leaf removal:
at each step, sum the leaf index reducing the sum to a telescoping product of row
sums $\sum_{j} s_{ij}$.
Hence the label sum over all (not necessarily injective) labelings
of $\tau$ equals
\[\sum_{i_1, \cdots  ,i_{|V|}}\overline{x_{i_1}}\,y_{i_1}\prod_{(z,w) \in E}s_{i_zi_{w}}= \langle x, y \rangle+ O_{t}(\varepsilon).
\]
Restricting to injective labelings changes this by $O_{t}(d^{2})$:
each identification of two distinct vertices introduces at least one
extra variance factor bounded by $\max_{i,j} s_{ij} \leq d^{2}$. Since $d \le 1$, the
total contribution of the leading shapes is
\[
m_{t}^{\mathrm{sc}}\cdot \big(\langle x, y \rangle + O_{t}(\varepsilon + d)\big).
\]

\textit{Step 3: Non-leading shapes.} Every non-leading shape satisfies at least one of the
following: (i)~some edge has multiplicity $k_e \geq 3$, or
(ii)~$G$ contains a cycle.
Fix a spanning tree $T$ of $G$. In each case, $t-2|E(T)| \ge 1$. Using the mixed-moment bound $s_{ij}d^{k_e-2}$, retain one variance factor $s_{i_zi_w}$ on each edge of $T$
and bound every remaining variance factor by $d^{2}$. The absolute contribution of a fixed shape is bounded by
\[d^{\,t-2|E(T)|}\cdot\sum_{i_1\neq \cdots \neq i_{|V|}}|x_{i_1}|\,|y_{i_\alpha}|\prod_{(z,w) \in E(T)}s_{i_zi_{w}} \le d\cdot\sum_{i_1\neq \cdots \neq i_{|V|}}|x_{i_1}|\,|y_{i_\alpha}|\prod_{(z,w) \in E(T)}s_{i_zi_{w}} . \]

For the absolute label sum over the spanning tree $T$,
the endpoint factor is $|x_{i_1}|\,|y_{i_\alpha}|$.
When $\alpha = 1$, successive leaf removal reduces the
tree sum to a product of row sums
$\max_i \sum_j s_{ij} \leq 1+\varepsilon \leq 2$,
leaving only the endpoint factor $\sum_{i_1}|x_{i_1}y_{i_1}| \leq 1$.
When $\alpha \neq 1$, first remove the branches off the path joining the endpoints, using the row-sum bound. The remaining path sum is bounded by
$\langle |x|, S^l |y|\rangle \leq \|S\|_{\mathrm{op}}^l \leq (1+\varepsilon)^l$,
where $S=(s_{ij})$, and $l$ is the path length between vertices $1$ and $\alpha$ in $T$.
Thus, in both cases the label sum is at most $(1+\varepsilon)^{|V|-1} \leq 2^{|V|-1} \le 2^{t-1}$.
Since $t$ is fixed, there are finitely many shapes, thus the contribution
of all non-leading shapes is bounded by $C_t d$.

Combining the two contributions, we conclude that
\[
  \Bigl|\langle x,\,(\mathbb{E}[B^{t}] - m_{t}^{\mathrm{sc}} I_{N})\,y
  \rangle\Bigr|
  \le C_{t}(\varepsilon + d),
\]
and thus completes the proof.
\end{proof}

The next concentration lemma follows from the Efron--Stein moment inequality of~\cite[Theorem~2]{BBLM}.

\begin{lemma}[Concentration by entry resampling]\label{lem:bounded-entry-concentration}
Let $B=(b_{ij})_{i,j\leq N}$  be a real symmetric or complex Hermitian random matrix whose upper-triangular entries are independent and centered. Suppose
\[
\max_i\sum_j\E\big[|b_{ij}|^2\big]\leq2,
 \qquad |b_{ij}|\leq d\leq1\quad\text{almost surely}.
\]
Then for any deterministic unit vectors $x,y$, every integer $t\geq1$, and every real $q\geq2$,
\begin{equation}\label{eq:bounded-entry-concentration}
 \norm{\bigp{x}{B^t y}-\E\Big[\bigp{x}{B^t y}\Big]}_{L^q}
 \leq C_t d\sqrt q\,
 \norm{(1+\norm{B}_{\op})^{t-1}}_{L^q},
\end{equation}
where $C_t$ is a constant depending only on $t$.
\end{lemma}

\begin{proof}
   For $e = (i,j)$ with $1 \leq i \leq j \leq N$, let $B^{(e)}$ be
the matrix obtained from $B$ by replacing $b_{ij}$ with an
independent copy $b'_{ij}$, together with its conjugate entry.
Set $\Delta_e := B^{(e)} - B$, and let $P_e$ be the orthogonal
projection onto the coordinates indexed by $\{i,j\}$ (onto
coordinate $i$ alone if $i = j$). Then
\[
  \Delta_e = P_e \Delta_e P_e,
  \qquad
  \|\Delta_e\|_{\op} = |b'_{ij} - b_{ij}| \leq 2d.
\]
Denote by $\mathbb{E}'$ the conditional expectation with respect
to the replacement entry $b'_{ij}$, and set
\[
w_e:=\mathbb E'|b'_{ij}-b_{ij}|^2
=|b_{ij}|^2+\mathbb E|b_{ij}|^2.
\]
A direct computation gives
\begin{equation}\label{eq:resampling-weights}
  w_e \leq 2d^2,
  \qquad
  \sum_e w_e P_e \leq (2 + \|B\|_{\op}^2) I_N.
\end{equation}
Indeed, $\sum_e w_e P_e$ is a diagonal matrix with $i$-th diagonal
entry $\sum_j |b_{ij}|^2 + \sum_j \mathbb{E}|b_{ij}|^2$.
The second bound follows from $\sum_j |b_{ij}|^2 = (B^2)_{ii}
\leq \|B\|_{\op}^2$, together with the variance row-sum assumption.

We now bound the difference   $(B^{(e)})^t - B^t$. Expanding it, we obtain
\[
  (B^{(e)})^t - B^t
  =  (B + \Delta_e)^t - B^t
  = \sum_{a=1}^{t}
    \sum_{\substack{k_0,\ldots,k_a \geq 0 \\ k_0+\cdots+k_a = t-a}}
    B^{k_0} \Delta_e B^{k_1} \cdots \Delta_e B^{k_a}.
\]
For $a = 1$, a direct computation using the definition of $\Delta_e$
and the bound $w_e \leq 2d^2$ gives, for any vectors $u, v$,
\[
  \sum_e \mathbb{E}'\bigl|\langle u, \Delta_e v \rangle\bigr|^2
  \leq C d^2 \|u\|_2^2 \|v\|_2^2.
\]
Taking $u = B^{k_0} x$ and $v = B^{k_1} y$, the right hand side
is bounded by $C d^2 \|B\|_{\op}^{2t-2}$. For $a \geq 2$, set $u = B^{k_0} x$ and $v = B^{k_a} y$.
Since $\Delta_e = P_e \Delta_e P_e$ and $\|\Delta_e\|_{\op} \leq 2d$,
\begin{equation*} \begin{aligned}\bigl|\langle u,\, \Delta_e B^{k_1} \cdots B^{k_{a-1}} \Delta_e v
  \rangle\bigr|
 & \leq \|\Delta_e\|_{\op}^{a} \|B\|_{\op}^{\sum_{j=1}^{a-1} k_j}
  \|P_e u\|_2 \|P_e v\|_2\\
 & \leq (2d)^{a-1} |b'_{ij} - b_{ij}|
  \|B\|_{\op}^{\sum_{l=1}^{a-1} k_l}
  \|P_e u\|_2 \|P_e v\|_2.
\end{aligned}
\end{equation*}

Therefore,
\begin{align*}
  &\sum_e \mathbb{E}'\!\left[
    \bigl|\langle x, B^{k_0} \Delta_e B^{k_1} \cdots
    \Delta_e B^{k_a} y \rangle\bigr|^2
  \right] \\
  &\quad \leq (2d)^{2a-2}
    \|B\|_{\op}^{2\sum_{l=1}^{a-1} k_l}
    \sum_e w_e \|P_e u\|_2^2 \|P_e v\|_2^2 \\
  &\quad \leq (2d)^{2a-2} \bigl(2 + \|B\|_{\op}^2\bigr)
    \|B\|_{\op}^{2(t-a)} \\
  &\quad \leq C_t\, d^2 \bigl(1 + \|B\|_{\op}\bigr)^{2t-2},
\end{align*}
where the second inequality uses \eqref{eq:resampling-weights}
and $\|P_e v\|_2 \leq \|v\|_2$. Cauchy--Schwarz over the finitely many words yields
\[
 \E'\left[\sum_e
 \left|\ip{x}{B^t y}-\ip{x}{(B^{(e)})^t y}\right|^2\right]
 \leq C_t d^2(1+\norm{B}_{\op})^{2t-2}.
\]
This bounds the one-sided resampling proxies for both the real and
imaginary parts of $\langle x, B^t y \rangle$.
Applying \cite[Theorem~2]{BBLM} to each part and combining them via
the $L^q$ triangle inequality completes the proof.
\end{proof}

Combining the two lemmas with \cref{lem:bounded-Hermitian-moment} gives the following estimate.

\begin{proposition}\label{prop:bounded-isotropic}
Let $B_N$ be a real symmetric or complex Hermitian random matrix whose upper-triangular entries are independent and centered.  Suppose
\[
 \varepsilon_N:=\max_i\Big|\sum_j\E\big[|(B_N)_{ij}|^2\big]-1\Big|\to0,
 \qquad |(B_N)_{ij}|\leq d_N\quad\text{almost surely},
 \qquad d_N\sqrt{\log(eN)}\to0.
\]
Then, for every fixed integer $t \geq 0$ and all $\eta, D > 0$,
and for all sufficiently large $N$,
\begin{equation}\label{eq:bounded-core-isotropic-tail}
 \sup_{\norm{x}_2=\norm{y}_2=1}
 \Pp\left(\big|\bigp{x}{B_N^t y}
       -m_t^{\rm sc}\ip{x}{y}\big|>\eta\right)
 \leq N^{-D}.
\end{equation}
Here the supremum is over deterministic unit vectors. Consequently, the polynomial isotropic law \eqref{eq:IGL} holds  for $B_N$ almost surely.
\end{proposition}

\begin{proof}
The case $t = 0$ is trivial, so we assume henceforth that $t \geq 1$. For every fixed $A > 0$, applying \cref{lem:bounded-Hermitian-moment}
with $k_N = \lceil A \log(eN) \rceil$ yields
\[
 \left(\E\big[\Tr B_N^{2k_N}\big]\right)^{1/(2k_N)}
 \leq2N^{1/(2k_N)}\sqrt{1+\varepsilon_N}
      +C\sqrt{k_N}\,N^{1/k_N}d_N=O_A(1).
\]
Set $q_N=\max\{2,\lceil\log(eN)\rceil\}$. 
Triangle inequality of the $L^{q_N}$ norm and the moment bound above yield
\[\norm{(1+\norm{B_N}_{\op})^{t-1}}_{L^{q_N}} \le C_t\sum_{j=0}^{t-1}\left(\E\|B_{N}\|^{jq_{N}}_{\op} \right)^{1/{q_{N}}}\le C_t\sum_{j=0}^{t-1}\left(\E\big[\Tr B_N^{2jq_{N}}\big]\right)^{1/(2q_N)}=O_{t}(1). \]
Then \cref{lem:bounded-entry-concentration} gives
\[ \sup_{\norm{x}_2=\norm{y}_2=1}
 \norm{\bigp{x}{B_N^t y}-\E\Big[\bigp{x}{B_N^t y}\Big]}_{L^{q_N}}\le C_td_N\sqrt{q_N} =o(1).
\]
For fixed $\eta, D$, it follows from \cref{lem:bounded-core-bias}
that for all sufficiently large $N$,
\[\Big|\E\Big[\bigp{x}{B_N^t y}\Big]-m_t^{\rm sc}\ip{x}{y}\Big|
 \leq C_t(\varepsilon_N+d_N) \le \frac{\eta}{2}.
\]
The bound~\eqref{eq:bounded-core-isotropic-tail} then follows from
Markov's inequality, and the last assertion follows from the
Borel--Cantelli lemma.
\end{proof}

\subsection{Application to the two entry regimes}

The following proposition is the counterpart of the isotropic global law in~\cite[Proposition~3.6]{Au} for the variance profiles and entry regimes considered here.

\begin{proposition}[Polynomial isotropic law]\label{prop:IGL}
Under (\ref{ass:FM}), for every fixed polynomial $f$ and every pair of deterministic unit-vector sequences $\{x_N\},\{y_N\}$,
\begin{equation}\label{eq:IGL-two-regimes}
 \ip{x_N}{f(H_N)y_N}-\ip{x_N}{y_N}\musc(f)\xtoP0.
\end{equation}
Under \eqref{ass:SW}, this convergence holds almost surely.
\end{proposition}

\begin{proof}
\emph{Finite moments.} The centered core $B_N$ in \eqref{eq:finite-cutoff} satisfies $|(B_N)_{ij}|\leq2/\log(eN)$ and $\varepsilon_N\to0$ by \eqref{eq:finite-row-defect}. Hence \cref{prop:bounded-isotropic} applies to $B_N$. By \cref{lem:tail-matching}, $\norm{H_N-B_N}_{\op}\xtoP0$. Hence, by Lemma \ref{prop:finite-edge}, both $\|H_N\|_{\op}$ and $\|B_N\|_{\op}$ are bounded in probability. So, for each fixed $t\geq1$,
\[
 \norm{H_N^t-B_N^t}_{\op}
 \leq\norm{H_N-B_N}_{\op}
       \sum_{j=0}^{t-1}\norm{H_N}_{\op}^{j}
                       \norm{B_N}_{\op}^{t-1-j}
 \xtoP0.
\]
Linearity then yields \eqref{eq:IGL-two-regimes}.

\emph{Sub-Weibull entries.} Set
\[
 \Lambda_N:=\max\Big\{1,
       \bigl(\sigstar\sqrt{\log(eN)}\bigr)^{-1/2}\Big\}.
\]
Since $\alpha\leq2$, (\ref{ass:SW}) gives $\sigstar\sqrt{\log(eN)}\to0$. Hence $\Lambda_N\to\infty$ and $\sigstar\Lambda_N\sqrt{\log(eN)}\to0$. Define the truncated Hermitian core $U_N$ by
\[
 (U_N)_{ij}:=\sigma_{ij}\left(
 W_{ij}\1_{\{|W_{ij}|\leq\Lambda_N\}}
 -\E[W_{ij}\1_{\{|W_{ij}|\leq\Lambda_N\}}]\right),\qquad i\leq j,
\]
with the lower-triangular entries determined by Hermitian symmetry. Its upper-triangular entries remain independent and satisfy $|(U_N)_{ij}|\leq2\sigstar\Lambda_N$. Integration of \eqref{eq:SW-tail} and the variance-profile identity give
\[
 \max_i\Big|\sum_j\E\big[|(U_N)_{ij}|^2\big]-1\Big|
 \leq C\bigl((\sigstar)^2+e^{-c\Lambda_N^\alpha}\bigr)\to0.
\]
Thus, \cref{prop:bounded-isotropic} applies to $U_N$. Let $C_N$ be the off-diagonal part of $H_N-U_N$ and put $k_N=\lceil\log(eN)\rceil$. The centered tail estimates \eqref{eq:ASTY-tail-symmetrization}--\eqref{eq:ASTY-tail-part}, with $L=\Lambda_N$, yield
\[
\Big(\E\big[\Tr C_N^{2k_N}\big]\Big)^{1/(2k_N)}
 \leq C N^{1/(2k_N)}
       \bigl(e^{-c\Lambda_N^\alpha}+\sigstar k_N^{1/\alpha}\bigr)=o(1).
\]
Markov's inequality gives $\Pp(\norm{C_N}_{\op}>u)\leq N^{-D}$ for every fixed $u,D>0$ and all sufficiently large $N$. The diagonal error is at most $\norm{\operatorname{diag}H_N}_{\op}+2\sigstar\Lambda_N$; \eqref{eq:diag-SW} gives the same bound for this term. Consequently,
\begin{equation}\label{eq:core-approximation-superpoly}
 \Pp(\norm{H_N-U_N}_{\op}>u)\leq N^{-D}
\end{equation}
for every fixed $u,D>0$ and all sufficiently large $N$.

Taking $D = 2$ and intersecting over a countable sequence $u \downarrow 0$,
the Borel--Cantelli lemma gives $\|H_N - U_N\|_{\op} \to 0$ almost surely.
By \cref{prop:SW-edge}, $\|H_N\|_{\op} \to 2$ almost surely,
so $\|U_N\|_{\op}$ is eventually bounded. The same telescoping bound thus gives $\norm{H_N^t-U_N^t}_{\op}\to0 $ almost surely for every fixed $t\geq1$. Combining this with \cref{prop:bounded-isotropic} applied to $U_N$ and using linearity completes the proof.
\end{proof}

\subsection{Proof of \texorpdfstring{\cref{thm:main}}{the main theorem}}
\begin{proof}
Under \eqref{ass:FM}, \cref{prop:finite-edge,lem:as-semicircle,prop:IGL} yield \eqref{eq:SC}--\eqref{eq:IGL} in probability. %
Under \eqref{ass:SW}, \cref{prop:SW-edge,lem:as-semicircle,prop:IGL}
yield \eqref{eq:SC}--\eqref{eq:IGL} almost surely. The proof is completed by applying \cref{thm:abstract-transfer}
in the corresponding mode of convergence.
\end{proof}

\section*{Acknowledgments}
This work was supported by the National Natural Science Foundation of China
under Grants \#12371157. We would like to thank Zhonggen Su and Jun Yin for valuable discussions, and Ramon van Handel for useful discussions on inhomogeneous random matrices and matrix concentration inequalities.

\appendix
\section{Sharpness and convergence modes}\label{sec:finite-order}

We use block diagonal  matrices to show that the profile conditions are sharp and that the finite-moment assumptions do not in general give almost-sure convergence. Related block-diagonal constructions can be found in \cite[Theorem~1.3 and Section~3]{ASTY}.  Let $1\leq d_N\leq N$ be integers with $d_N\to\infty$. Partition the coordinates into $\lfloor N/d_N\rfloor$ blocks: each of size $d_N$, except possibly the last, whose size lies between $d_N$ and $2d_N-1$. Set
\begin{equation}\label{eq:block-profile}
 \sigma_{ij}=
 \begin{cases}
  m^{-1/2},&i,j\text{ belong to the same block of size }m,\\
  0,&i,j\text{ belong to different blocks}.
 \end{cases}
\end{equation}
Note that, by construction, $\sum_j \sigma_{ij}^2 = 1$ and $\sigstar = d_N^{-1/2}$.

\begin{example}[Sub-Weibull]\label{ex:SW-sharpness}
For every $0<\alpha\leq2$, there exists a real symmetric matrix $H_N$ satisfying (\ref{ass:noise})--(\ref{ass:profile})  and \eqref{eq:SW-norm} such that
\[
 0<\lim_{N\to\infty}\sigstar(\log N)^{1/\alpha}<\infty,
 \qquad
 \liminf_{N\to\infty}\norm{H_N}_{\op}\geq3
 \quad\text{almost surely}.
\]
Thus the condition \eqref{eq:SW-scale} is sharp, even for convergence in probability.
\end{example}

\begin{proof}
Let $\xi_{\alpha}$ be a real symmetric Weibull variable, normalized by $\E[\xi_{\alpha}^2]=1$, and let $c_\alpha>0$ be the constant such that
\[
 \Pp(|\xi_{\alpha}|>t)=\exp(-c_\alpha t^\alpha),\qquad t\geq0.
\] %
We choose $b>0$ small enough such that $c_\alpha3^\alpha(2b)^{\alpha/2}\leq1/2$, and consider the block diagonal variance profile defined in \eqref{eq:block-profile} with
\[
 d_N=\left\lfloor b(\log N)^{2/\alpha}\right\rfloor.
\]
For all sufficiently large $N$, let the diagonal entries $W_{ii}$ be i.i.d. copies of a symmetric Bernoulli random variable taking values in $\{-\sqrt{2},\sqrt{2}\}$, and take the off-diagonal upper-triangular entries $\{W_{ij}:i<j\}$ to be i.i.d. copies of $\xi_\alpha$, independent of the diagonal. Consider $H_N$ as in \eqref{eq:model-H}. Then (\ref{ass:noise}) and  \eqref{eq:SW-norm} hold, and
\[
 \lim_{N \to \infty}\sigstar(\log N)^{1/\alpha}=b^{-1/2}.
\]
For any off-diagonal entry in a block of size $m<2d_N$,
\[
\Pp\big(|H_{ij}|>3\big)=\Pp\big(|\xi_\alpha| \ge 3\sqrt{m}\big)
 =\exp\!\big(-c_\alpha3^\alpha m^{\alpha/2}\big)
 \geq N^{-1/2}.
\]
By independence and the bound $\norm{H_N}_{\op}\geq\max_{i,j}|H_{ij}|$, we thus obtain
\[
\Pp\big(\norm{H_N}_{\op}\leq3\big)
 \leq \prod_{i<j}\Pp\big(|H_{ij}| \le 3\big) \le \exp\!\Big(-\frac{d_N-1}{2}\sqrt{N}\Big).
\]
Since the bound is summable, the Borel-Cantelli lemma implies that $\norm{H_N}_{\op}>3$  almost surely.
\end{proof}

\begin{example}[Finite moments]\label{ex:FM-sharpness}
The following real symmetric matrix models satisfy \eqref{ass:noise}--\eqref{ass:profile}, with entry distributions that may depend on 
$N$.
\begin{enumerate}[label=(\roman*),ref=\roman*,leftmargin=*]
\item\label{ex:FM-sharpness-scale} 

Fix $p\geq4$. For any sequence $h_N\to\infty$ with $h_N=o(N)$, there exists a model satisfying \eqref{eq:FM-UI} such that 
\[
 \sigstar\longrightarrow0,
 \qquad N(\sigstar)^{p-2}\sim h_N,
 \qquad
 \liminf_{N\to\infty}\Pp \big(\norm{H_N}_{\op}\geq3\big)
 \geq1-e^{-1/2}>0.
\]
Thus, the condition \eqref{eq:FM-scale} cannot, in general, be weakened to allow $N(\sigstar)^{p-2}\to \infty$, regardless of how slowly this quantity diverges.

\item\label{ex:FM-sharpness-as} There exists a model with variance profile $\sigma_{ij}=N^{-1/2}$ satisfying \eqref{ass:FM} with 
$p=4$ such that, almost surely,
\[
 \limsup_{N\to\infty}\bigp{e_1}{H_N^2e_1}\geq9,
 \qquad
 \limsup_{N\to\infty}\norm{H_N}_{\op}\geq3,
\]
where $e_1$ is the first coordinate vector. Thus neither the degree-two isotropic law \cref{prop:IGL} nor the norm limit \cref{prop:finite-edge} holds almost surely.
\end{enumerate}
\end{example}

\begin{proof}
For all sufficiently large $N$, take $0<\tau_N\leq1$. For an off-diagonal entry in a block of size $m$, put
\[
 \delta=\frac{\tau_N}{Nm},\qquad
 u=\left(\frac{1-9\tau_N/N}{1-\delta}\right)^{1/2},
\]
For each block of size $m$, we choose the off-diagonal upper-triangular entries
$W_{ij}$, $i<j$, in the block independently according to the distribution
\[
 \Pp(W_{ij}=3\sqrt m)=\Pp(W_{ij}=-3\sqrt m)=\frac\delta2,
 \qquad
 \Pp(W_{ij}=u)=\Pp(W_{ij}=-u)=\frac{1-\delta}{2}.
\]
Take independent symmetric $\{-\sqrt2,\sqrt2\}$ diagonal entries and independent symmetric $\{-1,1\}$ off-diagonal upper-triangular entries between different blocks, all independent of the within-block entries. 
We then consider $H_N$ as in \eqref{eq:model-H}. Note that each $W_{ij}$ is centered, and for each within-block off-diagonal entry,
\[
 \E[W_{ij}^2]=9m\delta+(1-\delta)u^2=1.
\]
Moreover, the rare atoms contribute
\[
 \delta(3\sqrt m)^p
 =\frac{3^p\tau_N m^{(p-2)/2}}{N}
 \leq\frac{3^p\tau_N(2d_N)^{(p-2)/2}}{N}
\]
to the $p$th moment, and all other atoms are bounded by $\sqrt2$. Thus \eqref{eq:FM-UI} holds whenever
\begin{equation}\label{equ:uniform-integrable}
 \frac{\tau_N d_N^{(p-2)/2}}{N}\longrightarrow0.
\end{equation}

For part~\textup{(\ref{ex:FM-sharpness-scale})}, choose
\[
 \tau_N=1,\qquad
 d_N=\left\lfloor (N/h_N)^{2/(p-2)}\right\rfloor.
\]
Then $d_N\to\infty$, $d_N=o(N)$, and $N(\sigstar)^{p-2}\sim h_N$, so \eqref{equ:uniform-integrable} holds. Let $E_N$ be the event that at least one rare off-diagonal atom occurs within a block. The sum of the rare-atom probabilities is
\[
 \sum_{\text{blocks of size }m}\frac{m(m-1)}{2Nm}
 =\frac{N-\lfloor N/d_N\rfloor}{2N}
 \longrightarrow\frac12,
\]
and the largest individual probability is at most $1/(Nd_N)\to0$. Independence and $\log(1-x)=-x+O(x^2)$ give
\[
 \Pp(E_N^c)
 =\prod_{\text{blocks of size }m}
   \left(1-\frac1{Nm}\right)^{m(m-1)/2}
 \longrightarrow \frac{1}{\sqrt{e}}.
\]
On $E_N$, at least one entry of $H_N$ has modulus $3$, so $\norm{H_N}_{\op}\geq3$. This proves part~\textup{(\ref{ex:FM-sharpness-scale})}.

For part~\textup{(\ref{ex:FM-sharpness-as})}, take $p=4$, $d_N=N$, and $\tau_N=1/\log(eN)$ (so the profile is flat and $N(\sigstar)^2=1$). Realize the matrices $H_N$ independently across $N$. %
Then \eqref{equ:uniform-integrable} is satisfied, and (\ref{ass:FM}) holds with $p=4$. Let $F_N$ be the event that the first row contains a rare off-diagonal entry.  On $F_N$, we have
\[
 \bigp{e_1}{H_N^2e_1}=\sum_j|H_{1j}|^2\geq9,
 \qquad
 \norm{H_N}_{\op}\geq\norm{H_Ne_1}_2\geq3.
\] 
Moreover, a direct computation gives
\[
 \Pp(F_N)=1-\left(1-\frac1{N^2\log(eN)}\right)^{N-1}
 \sim\frac1{N\log(eN)}.
\]
Hence, $\sum_N\Pp(F_N)=\infty$. The events $F_N$ are independent because the matrices $H_N$ are constructed independently across $N$. The second Borel--Cantelli lemma~\cite{Billingsley} yields $\Pp(F_N \text{ i.o.})=1$, and the desired conclusion follows.
\end{proof}

\begingroup
\small
\newcommand{\etalchar}[1]{$^{#1}$}

\endgroup

\end{document}